\documentclass[10pt]{article}
\usepackage[british]{babel}
\usepackage{mathrsfs}
\usepackage{amsmath}
\usepackage{amssymb}
\usepackage{amsthm}
\usepackage{enumerate}
\usepackage{epic}
\usepackage{color}

\numberwithin{equation}{section}
\makeatletter
\renewcommand{\subsection}{\@startsection
{subsection}{2}{0mm}{\baselineskip}{-0.25cm}
{\normalfont\normalsize\bf}}
\makeatother

\newtheorem{theorem}{Theorem}[section]
\newtheorem{proposition}[theorem]{Proposition}
\newtheorem{lemma}[theorem]{Lemma}
\newtheorem{corollary}[theorem]{Corollary}
\newtheorem{remark}[theorem]{Remark}
\newtheorem{result}[theorem]{Result}

\def\F{\mathbf F}

\def\cF{\mathcal F}

\def\cX{\mathcal X}

\def\xfq2{{\cX(\mathbb{F}_{q^2})}}
\def\F+xfq2{{\cF^+(\mathbb{F}_{q^2})}}
\def\xpfq2{{\cX^+(\mathbb{F}_{q^2})}}
\def\xmfq2{{\cX^-(\mathbb{F}_{q^2})}}
\title{The Structure of 2-Colored Best Match Graphs}
\author{Annachiara Korchmaros\footnote{Annachiara Korchmaros: Department of Mathematics and Systems Analysis, Aalto University, annachiara.korchmaros@aalto.fi}}

\begin{document}
\date{}
\maketitle
\begin{abstract}
Recent investigations in computational biology have focused on a family of 2-colored digraphs, called 2-colored best match graphs, which naturally arise from rooted
phylogenetic trees. Actually the defining properties of such graphs are unusual, and a natural question is whether they also have properties which well fit in structural graph theory. In this paper we prove that some underlying oriented bipartite graphs of a 2-colored best match graph are acyclic and we point out that the arising topological ordering can efficiently be used for constructing new families of 2-colored best match graphs.
\end{abstract}
\section{Introduction}
Let $\Gamma=\Gamma(V,E)$ be a bipartite digraph without loops and multiple edges. For a vertex $u$ of $\Gamma$, $v$ is an \emph{out-neighbor}  (\emph{in-neighbor}) of $u$ if $uv$ (respectively $vu$) is an edge of $\Gamma$. The set of all out-neighbors (in- neighbors) of $u$ is denoted by $N(u)$ and $N^-(u)$ respectively. A {\emph{ 2-colored best match graph}} (2-cBMG) is a bipartite graph such that $N(u)$ is not empty for any $u\in V$ and each of the following three properties holds for any two $u,v\in V$:
\begin{itemize} \label{Ns}
\item[{\bf{N1}}:] $u\cap N(v)=v\cap N(u)=\emptyset$ implies $N(u)\cap N(N(v))=N(v)\cap N(N(u))=\emptyset$.
\item[{\bf{N2}}:] {\mbox{$ N(N(N(u))) \subseteq N(u)$.}}
\item[{\bf{N3}}:]  $u\cap N(N(v))=v\cap N(N(u))=\emptyset$ together with $N(u)\cap N(v)\neq \emptyset$ implies $N^-(u)=N^-(v)$ and one of the inclusions $N(u)\subseteq N(v)$, $N(v)\subseteq N(u)$.
\end{itemize}
Before presenting our results we point out that such digraphs are strongly related to detection of ortologous genes;
see \cite{geiss,geiss1,geiss3}.
Let $T$ be a (rooted, phylogenetic) tree with leaf set $L$ and a surjective color-map $\sigma:L\rightarrow S$ for a non-empty color set $S$. Then $y\in L$ is a best match of $x\in L$, in symbols $x\rightarrow y$, if ${\rm{lca}}(x,y)\prec {\rm{lca}}(x,y')$ for all $y'\in L$ with $\sigma(y)=\sigma(y')$. Here ${\rm{lca}}$ stands for the last common ancestor, and the partial ordering $p\prec q$ occurs if $q$ is located above $p$ along the path connecting $p$ with the root of $T$. The associated colored best match graph (cBMG) is the directed graph on the vertex set $L$ where the arcs are the ordered pairs $xy$ with $x\rightarrow y$ and  $x\neq y$. It is a colored digraph with color map $\sigma$. If a colored digraph $G(T,\sigma)$ is isomorphic to the cBMG associated to the rooted tree $T$, then $T$ is said to explain the vertex-colored graph $(G,\sigma)$. In a cBMG, there is natural equivalence relation $\dot{\sim}$ where $x{\dot{\sim}}y$ if $x$ and $y$ have the same out-neighbors and in-neighbors. In particular, $x\dot{\sim}y$ implies that
$N(x)=N(y)$ but the converse is not always true. In the study of best match graphs, the most important case is $|S|=2$, since all cBMGs have induced subgraphs which are 2-cBMGs; see \cite[Theorem 9]{geiss}. The axiomatization of the concept of a colored best much graph is due to the authors of the pioneering work \cite{geiss} who showed that a 2-cBMG can equivalently be defined as a certain bipartite graph $\Gamma(V,E)$. In fact, for a connected 2-colored digraph with at least one out-neighbor for each of its vertex, \cite[Theorem 4]{geiss} states that there exists a (rooted phylogenetic) tree $T$ explaining $(G,\sigma)$ if and only if $(G,\sigma)$ has properties {\bf{N1}},{\bf{N2}}, and {\bf{N3}}.

It should be noticed however that {\bf{N1}} and {\bf{N3}} had not been considered in the literature on graph theory until the discovery of their links to evolutionary relatedness via phylogenetic trees. {\bf{N2}} was introduced, but only marginally studied, under the name ``bi-transitive'' property in the preprint \cite{das}.

This gives a motivation to consider graph-theoretic properties of ``structural type'' in digraphs satisfying at least one property between {\bf{N1}},{\bf{N2}}, and {\bf{N3}}.
We are mostly concerned with bi-transitivity, and our contributions are stated and proven in Section \ref{secpc}. It turns out that our results on circuits and paths fit in well with classical works in graph theory dating back to 1980's. In the discussion in Section \ref{secpc} we point out two important structural properties of bi-transitive digraphs, namely they has a long path, and some of the underlying oriented digraphs are acyclic. If we assume that {\bf{N1}} and {\bf{N3}} also hold then our results can be refined, but it remains unclear how deeply {\bf{N3}} can affect the structure of the digraph.

Since acyclic digraphs have topological ordering, the structure of a 2-cBMG can be locally investigated; see Section \ref{2cbmg}. Especially, the behavior of maximal and minimal vertices uncovers several new properties which can also be used as an efficient tool in classifying smaller 2-cBMGs; see Section \ref{ccs}.

\section{Background on digraphs and their paths and circuits}
In this paper $\Gamma$ stands for a digraph without loops and multiple edges. With the usual notation, $V$ is its vertex-set, $E$ is its edge-set where for $u,v\in V$,
the edge with tail $u$ and head $v$ is denoted by $uv$ or $[u,v]$. A {\emph{symmetric edge}} is a pair such that $uv,vu \in E$. A digraph $\Gamma$  with vertex set $\{1,\ldots, n\}$ is denoted by $\Gamma:=<n|[i_1,j_1],\ldots,[i_r,j_r]>$ with $E=\{[i_1,j_1],\ldots,[i_r,j_r]\}$. As in \cite{geisscor}, $\Gamma$ is called sink-free if each of its vertices has at least one out-neighbor.

A \emph{bipartite digraph} $\Gamma=\Gamma(V,E)$ is a digraph whose vertices can be divided into two disjoint sets $V_1$ and $V_2$ such that every edge connects a vertex in $V_1$ to one in $V_2$ or vice versa.
The two sets $V_1$ and $V_2$ may be thought of as a coloring of the graph with two colors where a \emph{coloring} is a labeling of the vertices with the two colors such that no two vertices sharing the same edge have the same color. We will refer to $V_1$ and $V_2$ as the color classes. A bipartite graph is \emph{balanced} if $|V_1|=|V_2|$.

For an ordered pair of vertices $(u,v)$ of $\Gamma$, $v$ is \emph{strongly connected} to $u$, if there exists a directed walk $u=u_0\rightarrow u_1 \rightarrow u_2 \rightarrow \cdots\rightarrow u_m \rightarrow u_{m+1}=v$ such that $u_iu_{i+1}$ is an edge of $\Gamma$ for every $0\le i \le m$. The \emph{length} of the directed walk is the number of edges in it, i.e. $m+1$. A \emph{directed trail} is a directed walk in which all edges are distinct. A \emph{directed path} is a directed trail in which all vertices are distinct. A \emph{directed circuit} is a non-empty directed trail in which the first and last vertices are repeated. A \emph{directed cycle} is a directed trail in which all vertices but the first and last are distinct. Two vertices $u$ and $v$ in $\Gamma$ are \emph{independent} if $uv\not \in E$ and $vu\not \in E$, i.e. $v\not \in N(u)$ and $u\not \in N(v)$. Any two equivalent vertices $u$ and $v$ of $\Gamma$ are independent, otherwise $u\in N(v)=N(u)$ would imply that $uu\in E$ contradicting the assumption that $\Gamma$ has no loops. If $\Gamma$ is sink-free, then any two equivalent vertices have the same color as they share a common out-neighbour.

A digraph is \emph{oriented} if $uv\in E$ implies $vu\not\in E$. We can derive a new graph $\tilde{\Gamma}$ from a bipartite digraph $\Gamma$ by keeping the same vertex set but eliminating all symmetric edges from $\Gamma$ such that if for some $u,v\in V$ both $uv$ and $vu$ are edges in $\Gamma$, then we keep exactly one of them. We stress that $\tilde{\Gamma}$ may have some vertex without out-neighbor, although it cannot contain isolated vertices. Clearly, $\tilde{\Gamma}$ is an oriented bipartite graph, and we call it an underlying oriented digraph of $\Gamma$. In the special case where $uv\in \tilde{\Gamma}$ if and only if $u'v'\in\tilde{\Gamma}$ for any $u'{\dot{\sim}}u$ and $v'{\dot{\sim}}v$, $\tilde{\Gamma}$ is called a $\dot{\sim}$ consistent underlying oriented digraph of $\Gamma$.

A
bipartite digraph $\Gamma(V,E)$ is \emph{bi-transitive} if for all vertices $u_1,u_2,v_1,v_2\in V$ with $u_1v_1,v_1u_2,u_2v_2 \in E$ we have $u_1v_2 \in E$. An oriented bipartite digraph $\Gamma(V,E)$ is a \emph{bitournament} if for any two vertices $u,v\in V$ with different colors, either $uv\in E$ or $vu\in E$. An oriented digraph has a \emph{topological vertex ordering} if its vertices can be labeled $u_1,u_2,\ldots,u_n$ such that for any edge $u_iu_j$ we have $i< j$. A vertex $u$ is \emph{minimal} if there is no vertex $v$ such that $v<u$ and $vu\in E$, that is, if $N^{-}(u)=\emptyset$; maximality for vertices is defined analogously.
A sufficient condition for an oriented digraph to have a topological ordering is to be acyclic, that is, there is no directed cycle in the digraph; see \cite{das}. Acyclic oriented digraphs are odd-even graphs; see \cite[Theorem 3.4]{das} and \cite{jensen}.  For a pair $(\mathcal{A},\mathcal{O})$ where $\mathcal{A}$ is a finite set of non-negative even integers and $\mathcal{O}$ is a set a positive odd integers, the associated \emph{odd-even} oriented digraph ${\vec{\mathcal{G}}}_{\mathcal{A}}(\mathcal{O})$ on vertex-set $\mathcal{A}$ has edge-set $E$ with $ab\in E$ when both $\frac{1}{2}(a+b)$ and $\frac{1}{2}(b-a)$ belong to $\mathcal{O}$. In fact, ${\vec{\mathcal{G}}}_{\mathcal{A}}(\mathcal{O})$ is an oriented bipartite graph $\Gamma(V_1,V,2,E)$ with color classes $V_1=\{a|a\equiv 0 \pmod{4},a\in \mathcal{A}\}$ and $V_2=\{a| a\equiv 2 \pmod{4}|,a\in \mathcal{A}\}$. Acyclic digraphs are also relevant in genomics data processing; see \cite{appl}.

The equivalence relation $\dot{\sim}$ gives rise to the \emph{quotient graph}  $\Gamma'(V',E')$ whose vertices are the equivalence classes and edges are defined as follows. Let $u'$ and $v'$ be two vertices of $\Gamma'$, then the ordered pair $u'v'$ is in $E'$ whenever $uv\in E$ for every $u\in u'$ and $v\in v'$.
In other words, $v'\in N(u')$ if and only if $v\in N(u)$ for every $u\in u'$ and $v\in v'$. One of the main properties of $\Gamma'$ is that it contains no two equivalent vertices. Several properties are shared between $\Gamma$ and its quotient $\Gamma'$ such as connectivity and being bipartite; see \cite{ACHK}. Furthermore, let $\tilde{\Gamma}$ be a $\dot{\sim}$ consistent underlying oriented digraph of $\Gamma$ with edge set $F$. The subgraph $\Delta$ of $\Gamma'$ with edges $u'v'$ where $uv\in F$ for every $u\in u'$ and $v\in v'$ is an underlying oriented graph of $\Gamma'$. If $\tilde{\Gamma}$ has a nontrivial cycle, then $\Delta$ also does. Therefore, if $\Delta$ is acyclic, then $\tilde{\Gamma}$ is also acyclic.

With the above notation, the defining properties of a 2-cBMG can be restated as follows. A 2-cMBG is a bipartite digraph with the following properties:
\begin{itemize}
\item[{\bf{N1}}:] if $u$ and $v$ are two independent vertices, then there exist no vertices $w,t$ such that $ut,vw,tw\in E$.
\item[{\bf{N2}}:] bi-transitive.
\item[{\bf{N3}}:] for any two vertices $u,v$ with a common out-neighbor, if there exists no vertex $w$ such that either $uw,wv\in E$, or $vw,wu\in E$, then they have the same in-neighbors and either all out-neighbors of  $u$ are also out-neighbors of $v$ or all out-neighbors of  $v$ are also out-neighbors of $u$.
\item[{\bf{N4}}:] sink-free.
\end{itemize}
We will also use the term ``almost 2-cBMG'' when each but at most one vertex has an out-neighbor.

The smallest connected 2-cBMGs are on $2$ vertices. They are isomorphic to $\Gamma(2):=<2\mid [1,2],[2,1]>$ and consist of a single symmetric edge. The connected 2-cBMGs on $3$ vertices are isomorphic to either $\Gamma_1(3):=<3\mid [1,3],[2,3],[3,2]>$ or $\Gamma_2(3):=<3\mid [1,3],[3,1],[2,3],[3,2]>$ where $\Gamma_2(3)$ contains two equivalent vertices, $1$ and $2$. The following example on $10$ vertices and color classes $\{1,2,3,4,5,6\}$, $\{7,8,9,10\}$ comes from \cite[Fig. 7]{geiss}:
\begin{equation}
\label{geex}
\begin{array}{lll}
\Gamma_{10}:=& <10\mid [1,7],[1,8],[1,9],[1,10], [2,8],[3,9],[4,10],[5,9],[6,9], \\
&{[5,10],[6,10],[7,1],[7,2],[7,3],[7,4],[7,5],[7,6],[8,2],[9,3],[10,4]>}.
\end{array}
\end{equation}
Removing the edge $[5,8]$ gives a 2-cBMG on $10$ vertices containing two equivalent vertices, namely $5$ and $6$. Furthermore, each of the vertices $1,2,3,4,7,8,9,10$ is the endpoint of a symmetric edge.

In a recent manuscript \cite{SSH}, 2-cBMGs are characterized as 2-colored digraphs containing no subgraph isomorphic to any of the following digraphs:
\begin{itemize}
\item a bipartite graph on four vertices $x_1,x_2,y_1,y_2$ with color classes $\{x_1,x_2\}$ and $\{y_1,y_2\}$ such that if $x_1y_1,y_2x_2,y_1x_2\in E$, then $x_1y_2\not\in E$.
\item a bipartite graph on four vertices $x_1,x_2,y_1,y_2$ with color classes $\{x_1,x_2\}$ and $\{y_1,y_2\}$ such that if $x_1y_1,y_1x_2,x_2y_2\in E$, then $x_1y_2\not\in E$.
\item a bipartite graph on five vertices $x_1,x_2,y_1,y_2,y_3$ with color classes $\{x_1,x_2\}$ and $\{y_1,y_2,y_3\}$ such that if $x_1y_1,x_2y_2,x_1y_3,x_2y_3\}$, then $x_1y_2, x_2y_1 \not\in E$.
\end{itemize}

Several papers give sufficient conditions for bipartite digraphs, in terms of the number of edges, to have cycles and paths with specified properties.
These conditions can be viewed as digraph versions or variants of similar conditions on undirected bipartite graphs which were widely studied since the 1980's; see \cite{am,am1,aj,ch,das,D,BJ,MMMM,MM,NV,wang,zang1,zang2}. We recall those which are related to the present investigation. The references are \cite{am,am1,ch,MMMM,zang1}.
\begin{result}
\label{resCMM} Let $\Gamma(V,E)$ be a bipartite digraph with color classes $V_1$ and $V_2$. Let $|V_1|=a,|V_2|=b$, $a\le b$, and $k=\min\{N(x)+N^-(x)|x\in U\cup V\}$.
\begin{itemize}
\item[(i)] If $|E|\ge 2ab-b+1$, then $\Gamma(V,E)$ has a cycle of length $2a$.
\item[(ii)] If $|E|\ge 2ab-(k+1)(a-k)+1$, then $\Gamma(V,E)$ has a cycle of length $2a$.
\item[(iii)] If $|E|\ge 2ab-k(a-k)+1$, then for any two vertices $u$ and $v$ which have different color, there is a path from $u$ to $v$ of length $2a-1$.
\item[(iv)] If $|E|\ge 2ab-a+2$, then for $u,v\in V$, any set of $a-1$ vertices is contained in a path of length at least $2(a-1)$ from $u$ to $v$ while for $u, v\in V$, there are paths from $u$ to $v$ and from $v$ to $u$ of every odd length $m$ with $3\le m \le 2a-1$.
\item[(v)] If $a\le 2k-1$, then $\Gamma(V,E)$ has a cycle of length $2a$, unless either $b>a=2k-1$ and $\Gamma(V,E)\cong\Gamma_1(a,b)$ or $k = 2$ and $\Gamma(V,E)\cong \Gamma_2(3,b)$.
\end{itemize}
\end{result}
\begin{result}
\label{zanghA} Let $\Gamma(V,E)$ be an oriented bipartite digraph whose vertex in-degree is at least $h\geq 0$ and out-degree is at least $k\geq 0$ for all vertices. Then $\Gamma(V,E)$ contains either a directed cycle of length at least $2(k+h)$ or a directed path of length at least $2(k+h)+3$.
\end{result}
\begin{result}\label{wangth}
Let $\Gamma(V,E)$ be a balanced directed bipartite graph with $|V_1|=|V_2|n\ge 2$. Suppose that $N(u)+N^-(u)+N(v)+N^-(v)>3n + 1$ for all $u\in U,v\in V$. Then $\Gamma=(V,E)$ contains two vertex-disjoint directed cycles of lengths $2n_1$ and $2n_2$, respectively, for any positive integer partition $n = n_l + n_2$.
\end{result}

A nice example of a bi-transitive bitournament arises from arithmetic, see \cite{das}. For a nonempty subset $S$ of natural numbers, let $\Gamma_S(S,E)$ be the digraph with vertex set $S$ such that $uv\in E$  if $u<v$ and $u$ and $v$ have opposite parity. If $U$ consists of all even numbers in $S$ while $V$ consists of all odd numbers in $S$, then $\Gamma_S(S,E)$ is bi-transitive bitournament. The importance of this example is due to the following characterization; see \cite[Theorem 2.5]{das}.
\begin{result}
\label{dastheorem2.5} Let $\Gamma(V,E)$ be a bitournament. Then the following properties are equivalent.
\begin{itemize}
\item[(i)] $\Gamma(V,E)$ is bi-transitive.
\item[(ii)] $\Gamma(V,E)$ has no directed cycle.
\item[(iii)] $\Gamma((V,E)\cong \Gamma_S(S,E)$ for some nonempty subset $S$ of natural numbers.
\end{itemize}
\end{result}

\begin{result}
\label{dastheorem2.9} A bitournament $\Gamma(V,E)$ with $|V_1| = |V_2| = 2m$  and $|N(u)|=|N^-(u)|$ for all $u\in V$ is not bi-transitive.
\end{result}

Bi-transitive bitournaments are examples of 2-cBMGs where both {\bf{N1}} and {\bf{N3}} hold trivially. Also, $\Gamma$ is a strongly connected 2-cBMG if and only if any underlying oriented digraph of $\Gamma$ is a bitournament.

\section{Paths and circuits in bi-transitive digraphs}
\label{secpc}
In this section $\Delta=\Delta(V,E)$ denotes a bi-transitive bipartite digraph.
Our goal is to prove several results, see Corollary and Proposition \ref{prop11}, \ref{prop4}, and \ref{dec1}, which will be useful for the study of 2-cBMGs.
\begin{lemma}
\label{lem1}  Let $(v_1,v_2,u_2)$ be a triple of vertices of $\Delta$, where $v_1,v_2$ have the same color while both $v_1u_2$ and $u_2v_2$ are edges of $\Delta$.
Then
\begin{itemize}
\item[(i)] $N(v_2)\subseteq N(v_1)$,
\item[(ii)] $N^-(v_1)\subseteq N^-(v_2)$.
\end{itemize}
\end{lemma}
\begin{proof} The case $v_1=v_2$ is trivial, therefore $v_1\neq v_2$ is assumed. Observe that $v_1u_2$ is an edge of $\Delta$ if and only if $u_2\in N(v_1)$. Therefore $N(u_2)\subseteq N(N(v_1))$. Also, $u_2v_2\in E$ means
 $v_2\in N(u_2)$. This gives
$N(v_2) \subseteq N(N(u_2))\subseteq N(N(N(v_1))).$ On the other hand $N(N(N(v_1))\subseteq N(v_1)$ by {\bf{N2}}. Therefore $N(v_2)\subseteq N(v_1)$ which is claim (i).

To show claim (ii), take any vertex $u_1$ from $N^-(v_1).$ Then $u_1v_1$ is and edge of $\Delta$. Since $v_1u_2$ is also an edge of $\Delta$, this yields that the triple $(u_1,u_2,v_1)$ satisfies the hypothesis of Lemma \ref{lem1}. From (i) applied to $(u_1,u_2,v_1)$ we have $N(u_2)\subseteq N(u_1)$. Since $v_2\in N(u_2)$, this yields $v_2\in N(u_1)$, that is, $u_1\in N^-(v_2)$ which proves claim (ii).
\end{proof}
\begin{lemma}
\label{lem2} Let $v_1$ and $v_2$ be distinct vertices of $\Delta$ with the same color. If there exist $u_1,u_2$ (not necessarily distinct) vertices in $\Delta$ such that the edge set $\Delta(E)$ of $\Delta$ has property
\begin{equation}
\label{eq2}
v_1u_2\in \Delta(E),\,v_2u_1\in \Delta(E),\,u_2v_2\in \Delta(E),\,u_1v_1\in \Delta(E),
\end{equation}
then $v_1$ and $v_2$ are equivalent vertices of $\Delta$.
\end{lemma}
\begin{proof}
Obviously, the color of $v_1$ is different from that of $u_1$ and $u_2$. Thus $u_1$ and $u_2$ have the same color.

From the first and the third inclusion in (\ref{eq2}), Lemma \ref{lem1} holds for $(v_1,v_2,u_2)$. Therefore, we have $N(v_2)\subseteq N(v_1)$ and $N^-(v_1)\subseteq N^-(v_2)$. Using the second and the forth
inclusions, Lemma \ref{lem1} holds for ($v_2$, $v_1$,$u_1)$. Therefore, $N(v_1)\subseteq N(v_2)$ and $N^-(v_2)\subseteq N^-(v_1)$. These four inclusions together yield $N(v_1)=N(v_2)$ and $N^-(v_1)= N^-(v_2)$ which proves Lemma \ref{lem2}.
\end{proof}
\begin{lemma}
\label{lem3} If $\Delta$ contains no two equivalent vertices, then $\Delta$ has no directed circuits of length four.
\end{lemma}
\begin{proof}
By way of a contradiction, let $v_1\rightarrow u_1\rightarrow v_2\rightarrow u_2\rightarrow v_1$ be a length $4$ directed circuit of $\Delta$. From Lemma \ref{lem4}, $v_1 \neq v_2$ and $u_1 \neq u_2$. Then $v_1$ and $v_2$, as well as $u_1$ and $u_2$, have the same color while $v_1$ and $u_1$ have different colors. Furthermore, $v_1u_1\in E)$, $u_1v_2\in E$, $v_2u_2\in E$, and $u_2v_1\in E$. Then (\ref{eq2}) holds whenever we switch $u_1$ and $u_2$, then the claim follows from Lemma \ref{lem2}.
\end{proof}

We show that Lemma \ref{lem3} is a particular case of a much stronger result.
\begin{proposition}
\label{prop11} If $\Delta$ contains no two equivalent vertices, then no directed circuit of $\Delta$ has length greater than $2$. In particular, every underlying directed graph of $\Delta$ is acyclic.
\end{proposition}
\begin{proof} Since $\Delta$ is a bipartite graph, if a directed circuit exists, it has even length. Therefore, by way of a contradiction, let $v_1\rightarrow u_1\rightarrow v_2\rightarrow u_2\rightarrow \cdots \rightarrow v_i\rightarrow u_i\rightarrow \cdots\rightarrow  v_n\rightarrow u_n\rightarrow v_1$ be a directed circuit of $\Delta$ with length $2n\ge 4$. Then the vertices $v_i$ have the same color, as well as the vertices $u_i$, where the color of $v_i$ and $u_i$ are different. Take two consecutive vertices with the same color, say
$v_i$ and $v_{i+1}$. Then Lemma \ref{lem1} applies to the triple $(v_i,v_{i+1},u_i)$ showing that $N(v_{i+1})\subseteq N(v_i)$ and $N^-(v_i)\subseteq N^-(v_{i+1})$. Since this holds true for any $i$, we have
\begin{equation}
\label{eq5}
N(v_n)\subseteq N(v_{n-1})\subseteq \cdots \subseteq N(v_2)\subseteq N(v_1).
\end{equation}
Since Lemma \ref{lem1} also applies to the triple $(v_n,v_1,u_n)$ we also have $N(v_1)\subseteq N(v_n)$. This together with (\ref{eq5}) yields
$N(v_1)=N(v_2)=\cdots= N(v_n)$.
Similarly, we have
\begin{equation}
\label{eq6}
N^-(v_1)\subseteq N^-(v_2)\subseteq \cdots \subseteq N^-(v_{n-1})\subseteq N^-(v_n).
\end{equation}
Applying Lemma \ref{lem1} to the triple $(v_n,v_1,u_n)$ gives $N^-(v_n)\subseteq N(v_1)$ which together with (\ref{eq6}) yields
$N^-(v_1)=N^-(v_2)=\cdots= N^-(v_n)$.
Therefore, the vertices $v_i$ are all equivalent, contradicting our hypothesis.
\end{proof}
\begin{corollary}
\label{cor1}
Let $\Delta$ be a bi-transitive digraph.
\begin{itemize}
\item[\rm(i)] Any $\dot{\sim}$ consistent oriented digraph of $\Delta$ is acyclic.
\item[\rm(ii)] If $\Delta$ contains no two equivalent vertices, the only directed cycles of $\Delta$ have length $2$ and they are induced by symmetric edges.
\end{itemize}
\end{corollary}
\begin{remark}
\label{rem10set}\em{
The hypothesis that $\Delta$ contains no two equivalent vertices in Corollary~\ref{cor1}(ii) cannot be dropped, as the following example shows. Let $$\Delta:=<6| [1,2],[1, 4], [1, 6], [2, 3], [2, 5],[3,4],[4,3], [3,6],[6,3],[4,5],[5,4],[5,6],[6,5]>$$ be the bipartite digraph with color classes $\{1,3,5\}$ and $\{2,4,6\}$.  $\Delta$ is bi-transitive (even a 2-cBMG), and $3\rightarrow 4 \rightarrow 5 \rightarrow 6 \rightarrow 3$ is a cycle of length $4$.}
\end{remark}
\begin{proposition}
\label{prop11a}
Let $\Delta$ be a bi-transitive digraph. If  $v_1\rightarrow u_1\rightarrow v_2\rightarrow u_2\rightarrow \cdots \rightarrow v_i\rightarrow u_i\rightarrow \cdots\rightarrow  v_n\rightarrow u_n\rightarrow v_1$ is a directed circle of length $2n$ with $n\ge 2$, then the sub-digraph of $\Delta$ with the vertex set $\{u_1,\ldots,u_n,v_1,\ldots,v_n\}$ is a complete bipartite digraph with color classes $\{u_1,\ldots,u_n\}$ and $\{v_1,\ldots,v_n\}$.
\end{proposition}
\begin{proof} From the proof of Proposition \ref{prop11}, the vertices $v_1,\ldots,v_n$ are equivalent. Since $u_1\rightarrow v_2\rightarrow u_2\rightarrow \cdots \rightarrow v_i\rightarrow u_i\rightarrow \cdots\rightarrow  v_n\rightarrow u_n\rightarrow v_1\rightarrow u_1$ is another directed circle, the vertices $u_1,\ldots,u_n$ are equivalent, as well. Since $u_1v_2,v_2u_2\in E$, this
implies $u_iv_j\in E$ and $v_iu_j\in E$ for every $1\le i,j\le n$, whence the claim follows.
\end{proof}
Example (\ref{geex}) shows that 2-cBMGs may contain symmetric edges. With the terminology of \cite{geiss3}, symmetric edges are the reciprocal best matches, and they model orthologous genes. Two genes belonging to two different species $A$ and $B$ are \emph{orthologous} if their last common ancestor was the speciation event that separated the lineages of $A$ and $B$. We cannot have genes $u_1,u_2$ from species $A$ and $v_1$ from species $B$ such that both $v_1,u_1$ and $v_1,u_2$ are couples of orthologous genes. On the other hand, the 2-cBMG given in Remark \ref{rem10set} contains vertices that are the endpoint of two symmetric edges, for example $3,4,5,6$.
Even though the definition of 2-cBMG does not exclude the existence of more symmetric edges with a common endpoint, using 2-cBMGs to improve algorithms for orthologous genes detection motivates the study of 2-cBMGs where no two symmetric edges have a common endpoint.

\begin{lemma}
\label{lem4} Assume that $\Delta$ contains no two equivalent vertices. Then for any vertex $u$ of $\Delta$ there exists at most one vertex $v$ of $\Delta$ such that both $uv$ and $vu$ are edges of $\Delta$.
\end{lemma}
\begin{proof}
By way of a contradiction, there are $u,v_1$ and $v_2$ vertices of $\Delta$ such that $uv_1, v_1u, uv_2, v_2u$ are edges of $\Delta$. Note that $v_1\neq v_2$ as $\Delta$ does not have multiple edges and that $u\neq v_1$ and $u\neq v_2$ as $\Delta$ does not have loops. Let $w\in N(v_1)$. Then $w\in N(N(u)) \subseteq N(N(N(v_2)))$. By {\bf{N2}}, $w\in N(v_2)$. Similarly, $w\in N(v_2)$ yields $w\in N(v_1)$. Let $w\in N^{-}(v_1)$, that is $v_1\in N(w)$. Then $v_2\in N(N(N(w)))\subseteq N(w)$ whence $w\in N^{-}(v_2)$. Therefore $v_1$ and $v_2$ have the same in- and out-neighbors, contradicting our hypothesis.
\end{proof}
As we have already observed, the 2-cBMG $\Delta_{10}$ given in (\ref{geex}) has two equivalent vertices. It may also be observed that each vertex in its color class $\{7,8,9,10\}$ is the endpoint of a symmetric edge. This shows that the converse of Lemma \ref{lem4} does not hold, that is, the constraint on a 2-cBMG to contain no two equivalent vertices is stronger than that of non-existence of symmetric edges with a common endpoint. However, the following result show that Proposition \ref{prop11} and the second claim in Corollary \ref{cor1} remain true for 2-cBMGs free from symmetric edges with a common endpoint.
\begin{lemma}
\label{lem4bis}
If $\Delta$ contains no symmetric edges with a common endpoint, then $\Delta$ has no directed circuits of length $\ge 4$.
\end{lemma}
\begin{proof}
From Proposition \ref{prop11a}, if $\Delta$ has a directed circuit of length $2n$ with $n\ge 2$, then it contains a sub-digraph $\Sigma$ isomorphic to a complete bipartite digraph on $2n$ vertices. This yields that a vertex of $\Delta'$ is the endpoint of $n$ symmetric edges of $\Sigma$, and hence at least $n$ symmetric edges of $\Delta$ contradicting our hypothesis.    \end{proof}
\begin{remark}
\em{Lemma \ref{lem4bis} shows that (ii) of Corollary \ref{cor1} still holds if the hypothesis of non-existence of  two equivalent vertices is replaced by the less stronger ``If no two symmetric edges of $\Delta$  have a common endpoint''.}
\end{remark}
The following lemma resembles \cite[Lemma 8]{geiss} for $\Delta$.
\begin{lemma}
\label{lem24} Assume that  $\Delta$ is sink-free. For any two vertices $u$ and $v$ of $\Delta$, if $N(u)\cap N(v)=\emptyset$, then $N(N(u))\cap N(N(v))=\emptyset$.
\end{lemma}
\begin{proof}
By way of a contradiction, there exists $t\in N(N(u))\cap N(N(v))$. By {\bf{N2}} we have  $N(t)\subseteq N(N(N(u)))\subseteq N(u)$. Applying {\bf{N2}} on $v$ and $t$, we also have $N(t) \subseteq N(v)$. Thus, $N(u) \cap N(v)$ is not the empty if $N(t)\neq \emptyset$. Since we assumed that all vertices in $\Delta$ have at least one out-neighbor, $N(u) \cap N(v) \neq \emptyset$, contradicting the hypothesis.
\end{proof}
 \begin{lemma}
\label{lem21}
 For any three vertices $u,v$ and $w$ of $\Delta$ with the same color, if $u$ and $v$ have no common out-neighbors, then either $u$ or $v$ is not strongly connected to $w$.
\end{lemma}
\begin{proof} By way of a contradiction, we have two directed walks $u=u_0\rightarrow u_1 \rightarrow u_2 \rightarrow \cdots\rightarrow u_m \rightarrow u_{m+1}=w $ and $u=v_0\rightarrow v_1 \rightarrow v_2 \rightarrow \cdots\rightarrow v_n \rightarrow v_{n+1}=w$. In particular, $u_m$ is reachable from $u$. From {\bf{N1}}, $u_m\in N(u)\cup N(N(u))$. Therefore,
either $u_m\in N(u)$ or $u_m\in N(N(u))$.  Since $u$ and $w$ have the same color different from the one of $u_m$, the case $u_m\in N(N(u))$ cannot actually occur. Thus $u_m\in N(u)$. As $u_m w$ is an edge of $\Delta$, it turns out that $w\in N(N(u))$. The same argument applies to $v$, therefore $w\in N(N(u))\cap N(N(v))$. But this contradicts Lemma \ref{lem24} as $u$ and $v$ have no common out-neighbors.
\end{proof}
\begin{lemma}
\label{lem22}  Assume that $\Delta$ sink-free. For any three vertices $u,v$ and $w$ of $\Delta$, if $u$ and $v$ have the same color but have no common out-neighbors, then either $u$ or $v$ is not strongly connected to $w$.
\end{lemma}
\begin{proof} By Lemma \ref{lem21}, we may assume that $w$ does not have the same color of $u$ and $v$. Using the same argument in the proof of Lemma \ref{lem21}, $u_m\in N(u)\cup N(N(u))$. Here $u,w$ and $u_m,w$ are pairs of vertices of different colors, hence $u_m$ and $u$ have the same color and the case $u_m\in N(N(u))$ must occur. Therefore {\bf{N2}} gives $w\in N(N(N((u)))\subseteq N(u)$. The same argument applies to $v$, then $w\in N(v)\cap N(u)$. Thus $w$ is an out-neighbor of both $u$ and $v$, contradicting one the hypotheses.
\end{proof}

\begin{lemma}
\label{lem33}  Let $u$ and $v$ be any two vertices of $\Delta$. If $u$ is out-dominated by $v$, and $u$ is strongly connected to a vertex $w$ of $\Delta$, then $v$ is also strongly connected to $w$.
\end{lemma}
\begin{proof} Let $u=u_0\rightarrow u_1 \rightarrow u_2 \rightarrow \cdots\rightarrow u_m \rightarrow u_{m+1}=w$ be a directed path. Then $u_1\in N(u)$. Furthermore, $u_1\in N(v)$ since $u$ is dominated by $v$. Therefore $v=v_0\rightarrow u_1 \rightarrow u_2 \rightarrow \cdots\rightarrow u_m \rightarrow u_{m+1}=w$ is a directed path form $v$ to $w$.
\end{proof}

The following lemma follows by repeatedly use of bi-transitivity.
\begin{lemma}
\label{lem12ag}
Let $u,v\in V$ be two distinct vertices. Then $u$ is strongly connected to $v$ if and only if
either $uv\in E$ (and $u,v$ have different colors), or there exists $w\in V$ such that
$uw,wv\in E$ (and $u,v$ have the same color).
\end{lemma}

Now we show some more features of bi-transitive digraphs which also satisfies one of {\bf{N1}} ad {\bf{N3}}. First we prove a consequence of {\bf{N1}}.
\begin{lemma}
\label{lemA7ag} Assume that {\bf{N1}} holds. Let $u,w\in V_1$ and $v,z \in V_2$ be four pairwise distinct vertices such that $wz$ is a symmetric edge. If $|N(u)|=|N(v)|=1$ and
$uz\in E$, then $vw\not\in E$.
\end{lemma}
\begin{proof} Clearly, $N(u)=\{z\}$ and hence $N(N(u))=N(z)$. Therefore, $v\cap N(u)=\emptyset$ and $w\in N(N(u))$. Now, assume on the contrary that $vw \in E$. Then $N(v)=\{w\}$ and $u\cap N(v)=\emptyset$. Furthermore,
$w\in N(N(u))\cap N(v)$. But this contradicts {\bf{N1}}.
\end{proof}
\begin{lemma}
\label{lem23} Let $\Delta$ be a bi-transitive graphs which also satisfies {\bf{N1}}. Assume  that $\Delta$ is sink-free.  For any three vertices $u,v$ and $w$  of $\Delta$, if $u$ and $v$ have different colors but they are independent, then either $u$ or $v$ is not strongly connected to $w$.
\end{lemma}
\begin{proof}
By way of a contradiction, both $u$ and $v$ are assumed to be strongly connected to $w$. W.l.g. we may suppose that $w$ and $v$ have the same color different from the one of $u$. The argument in the proof of Lemma \ref{lem22} applied to $u,w$ shows that $w\in N(u)$, while the argument in the proof of Lemma \ref{lem21} applied to $v,w$ gives $w\in N(N(v))$. Therefore $N(u)\cap N(N(v))\neq\emptyset$. Then $u$ and $v$ are not independent by {\bf{N1}}, contradicting one of the hypotheses.
\end{proof}
Lemma \ref{lem23} has the following consequence.
\begin{proposition}
\label{prop4} Let $\Delta$ be a bi-transitive digraph which also satisfies both {\bf{N1}} and  {\bf{N4}}. Let $u$ and $v$ be two independent vertices of $\Delta$ with no common out-neighbor. Then, for any vertex $w$, either $u$ or $v$ is not strongly connected to $w$.
\end{proposition}
\begin{proposition}
\label{dec1} Let $\Delta$ be a bi-transitive digraph which also satisfies both {\bf{N3}} and  {\bf{N4}}.
Let $u$ and $v$ be non-equivalent vertices in $\Delta$ with a common out-neighbor. If there is no a directed path of length $2$ from $u$ to $v$ or vice-versa, then at least one of them is not the endpoint of a symmetric edge.
\end{proposition}
\begin{proof}
By our hypotheses, $u$ and $v$ satisfy {\bf{N3}}. W.l.g. $u$ is dominated by $v$. Assume that there exists a vertex $w$ of $\Delta$ such that $uw,wu\in E$. Then $wv\in E$ by $N^{-}(u)=N^{-}(v)$. Furthermore $vw\in E$ since $u$ is dominated by $v$. Now the claim follows from Lemma \ref{lem4} applied to $w$.
\end{proof}
In Example (\ref{geex}), $N(5)=\{8,9,10\}$, $N(6)=\{9,10\}$, and  $N^{-1}(5)=N^{-1}(6)=\{1\}$. Therefore $6$ is dominated by $5$. This shows the hypotheses of Proposition \ref{dec1} are satisfied by $u=6$ and $v=5$. Accordingly, either $5$ or $6$ is not the endpoint of symmetric edge. Here actually both have that property.

The results reported in Section \ref{secpc} show that under natural hypotheses
$\tilde{\Delta}$ has longer directed paths. For instance, Result \ref{zanghA} yields that this occurs when $k+h$ is big enough. In Example (\ref{geex}), whenever we take a $\dot{\sim}$ consistent oriented digraph of $\Delta$, we always find two vertices one with no in-neighbors and the other with no out-neighbor, that is $k,h=0$. Result \ref{zanghA} ensures the existence of a directed path of length $3$ in $\tilde{\Delta}$ and hence in $\Delta$. Actually $\Delta$ contains the directed path $1\rightarrow 7 \rightarrow 5  \rightarrow 9 \rightarrow 3$ of length $4$.

\section{The structure of a 2-cBMG}
\label{2cbmg}
In this section $\Gamma$ stands for a (not necessarily connected) 2-cBMG.

Corollary \ref{cor1} has the following consequence.
\begin{corollary}
\label{cor11} Every $\dot{\sim}$ consistent underlying oriented digraph of a 2-cBMG is acyclic. If $\Gamma$ contains no two equivalent vertices, then the only directed cycles of a 2-cBMG have length $2$ and are induced by symmetric edges.
\end{corollary}
From Corollary \ref{cor11}, every $\dot{\sim}$ consistent underlying oriented digraph of $\Gamma$ has a topological ordering.
From now on, fix a $\dot{\sim}$ consistent underlying oriented digraph $\tilde{\Gamma}$ and label the vertices of $\Gamma$ by $1,2,\ldots,m$ so that if $ij$ is an edge of $\tilde{\Gamma}$, then $i<j$. In particular, $m$ is the last vertex of $\Gamma$.
\begin{proposition}
\label{prop23ag} Let $U$ be any set of minimal vertices in a topological ordering of a $\dot{\sim}$ consistent underlying oriented digraph of a 2-cBMG. Adding all edges with tails in $U$ to the edge-set produces a connected 2-cBMG.
\end{proposition}
\begin{proof} By induction on the size of $U$, it suffices to prove the claim for the case where $U$ consists of a single vertex $w$. Let $\hat{\Gamma}$ be the bipartite digraph obtained from $\Gamma(V,E)$ by adding all edges $wv$ to $E$ where $v$ ranges over $V\setminus \{w\}$. Obviously, $\hat{\Gamma}$ is connected and sink-free. Furthermore, $\hat{N}(w)=V_1$ with $V_1$ consisting of all vertices whose color is opposite to the color of $w$. For any vertex $u\neq w$ we have $\hat{N}(u)=N(u)$ and $\hat{N}^-(u)=N^-(u)\cup \{w\}$. From this, {\bf{N1}}, {\bf{N2}}, and  {\bf{N3}} follow for any two vertices $u,v$ distinct from $w$. Moreover, since $\hat{N}(w)=V_1$, {\bf{N2}} also holds for $w$. If $v$ is any vertex other than $w$, then {\bf{N1}} is trivial, and {\bf{N3}} is also trivial unless $v$ and $w$ have the same color but $v$ has no in-neighbor. In the latter case, $\hat{N}^-(w)=\hat{N}^-(v)=\emptyset$ and $\hat{N}(v)\subseteq V_1=\hat{N}(w)$, whence {\bf{N3}} follows.
\end{proof}

\begin{remark}\label{august24_remark1}
\em{Proposition \ref{prop23ag} shows that for any set of a pairwise disjoint 2-cBMGs  their join can be made into a connected 2-cBMG. This suggests that the number of pairwise non-isomorphic 2-cBMGs on $n$ vertices goes up rapidly.}
\end{remark}

\subsection{2-cBMGs containing no two symmetric edges with a common vertex}
\begin{proposition}
\label{propA7ag2020} With the above notation, the following claims hold.
\begin{itemize}
\item[\rm(i)] The vertex $m$ is the endpoint of a (unique) symmetric edge of $\Gamma$.
\item[\rm(ii)] If $\ell m$ is the symmetric edge of $\Gamma$ with vertex $m$, then the vertices $d\not\in \{\ell,m\}$ such that either $N(d)=\{m\}$, or $N(d)=\{\ell\}$ holds are equivalent. In particular, $N(d)=\{m\}$ and $N(e)=\{\ell\}$ do not hold simultaneously for any two vertices $d,e$.
\item[\rm(iii)] If $\ell<v$, then $\ell v \not\in E$.
\item[\rm(iv)] It is possible to rearrange the topological ordering of $\tilde{\Gamma}$ such that $\ell=m-1$ and, if $d_1,\ldots,d_r$ are all (equivalent) vertices with $N(d_1)=\ldots =N(d_r)=\{m\}$, then $d_1=m-2, \ldots d_r=m-1-r$ also holds.
\end{itemize}
\end{proposition}
\begin{proof} (i) Since $m$ is the last vertex in $\tilde{\Gamma}$, it has no out-neighbor in $\tilde{\Gamma}$. On the other hand, since $|N(m)|\geq 1$, there exists $\ell\in V$ such that $\ell\in N(m)$. Therefore, $\ell m\in E$ with $\ell<m$.  This shows that we removed $\ell m$ from $\Gamma$ when constructing $\tilde{\Gamma}$. Therefore, $m\ell$ was an edge we kept. Thus, both $\ell m$ and $m\ell$ are edges of $\Gamma$, that is, $\ell m$ is a symmetric edge. Therefore  $N(m)=\{\ell\}$.

(ii) Case A: {\em {There exist $d,e\in V$ distinct vertices such that $N(d)=N(e)=\{m\}$.}} In this case, $d$, $e$ and $\ell$ have the same color  $\Gamma$. We prove that the hypotheses in {\bf{N3}} hold for $u=d$ and $v=e$. To show that $d\cap N(N(e))=\emptyset$ assume on the contrary the existence of $w\in V$ for which both relations $d\in N(w)$ and $w\in N(e)$ hold. The latter one together with $N(e)=\{m\}$ yields $w=m$, and hence the former one reads $d\in N(m)$, that is, $md\in E$. Since $dm \in E$ also holds, this yields that $dm$ is a symmetric edge. From Lemma \ref{lem4}, this is only possible for $d=\ell$, a contradiction. Similarly, $e\cap N(N(d))=\emptyset$. Furthermore, $N(e)\cap N(d)=\{m\}$ and hence $N(e)\cap N(d) \ne \emptyset$. Thus {\bf{N3}} applies to $d,e$. Therefore  $N^-(d)=N^-(e)$. This together with $N(d)=N(e)=\{m\}$ show that $d$ and $e$ are equivalent vertices of $\Gamma$.

Case B: {\em{There exist $d,e\in V$ distinct vertices such that $N(d)=\{m\}$ and $N(e)=\{\ell\}$.}} This time $d$ and $\ell$ are in the same component of $\Gamma$, say $V_1$, and the other component $V_2$ contains $d$ and $m$. Since $dm \in E$, Lemma \ref{lemA7ag} applied to $u=d,v=e$ yields $e\ell\not\in E$, a contradiction as $\ell \in N(e)$.

Case C: {\em{There exist $d,e\in V$ distinct vertices such that $N(d)=N(e)=\{\ell\}$.}} Up to interchanging $m$ and $\ell$, the proof is the same as in Case A.

(iii) Assume on the contrary that $\ell v\in E$. Then $v$ and $\ell$ have different color, and hence $v$ and $m$ have the same color. Thus  $m\not\in N(v)$.
  Take a vertex $w\in N(v)$. Then $m\ell, \ell v,vw\in E$. From the bi-transitive property of $\Gamma$, we have that $mw\in E$. Since $m$ is the last vertex in $\tilde{\Gamma}$ this implies that $mw$ is a symmetric edge. From (i), $w=\ell$, a contradiction, contradicting the hypothesis that $m$ is the last vertex of $\tilde{\Gamma}$.

(iv) Take $v\in V$ such that $\ell<v<m$, and interchange $\ell$ and $v$ in the fixed topological ordering in $\tilde{\Gamma}$. Claim (iii) ensures that we obtain a new ordering which is still a topological ordering of $\tilde{\Gamma}$. Therefore, $\ell=m-1$ can always be assumed. Finally, assume that $d$, as given in Claim (ii), exists. Then $d<v$ with $dv\in E$ would imply $v=m$. We may suppose $v\neq m,l$ otherwise the claim holds. Thus, $dv\not\in E$. Hence interchanging $d$ with $v$ in the fixed topological order of the vertices of $\tilde{\Gamma}$ produces a new topological order of $\tilde{\Gamma}$. Therefore, $d_1=m-2, \ldots d_r=m-1-r$ can always be assumed.
\end{proof}

Now remove from $\Gamma$ both vertices $m$ and $\ell$. By Proposition \ref{propA7ag2020}, we may have some other vertices, say $d_1,\ldots, d_r$, such that either $N(d_i)=\{m\}$ for all $1\le i \le r$, or $N(d_i)=\{\ell\}$ for all $1\le i \le r$. If such vertices exist, we also remove all. Let $\bar{V}$ be the set of the remaining vertices, and $\bar{E}$ the edges of $\Gamma$ with both endpoints in $\bar{V}$. The arising bipartite digraph $\bar{\Gamma}=\bar{\Gamma}(\bar{V},\bar{E})$ is a \emph{truncated graph} of $\Gamma$. For $u\in \bar{V}$, we will denote the set of its out-neighbors in $\bar{\Gamma}$ by $\bar{N}(u)$, and that of its in-neighbors in $\bar{\Gamma}$ by $\bar{N}^-(u)$.
\begin{proposition}
\label{propB7ag2020} The truncated bipartite digraph $\bar{\Gamma}$ satisfies {\bf{N1}}, {\bf{N2}} and {\bf{N3}}. Its vertices without out-neighbors are pairwise equivalent as vertices of $\Gamma$.
\end{proposition}
\begin{proof} To prove {\bf{N1}} take $u,v\in \bar{\Gamma}$. We may assume that $u$ and $v$ have different color. From the hypotheses in {\bf{N1}}, we have  $u\not\in \bar{N}(v)$ and $v\not\in \bar{N}(u)$. Then $u\not\in N(v)$ and $v\not\in N(u)$, and  {\bf{N1}} applies to $u$ and $v$ in $\Gamma$. Therefore,  $N(u)\cap N(N(v))=N(v)\cap N(N(u))=\emptyset$. Since $\bar{N}(u)\subseteq N(u)$, $\bar{N}(v)\subseteq N(v)$, and  $\bar{N}(\bar{N}(u))\subseteq N(N(u))$, we obtain $\bar{N}(\bar{N}(u))\cap\bar{N}(v) \subseteq N(N(u))\cap N(v)=\emptyset$. Similarly, $\bar{N}(\bar{N}(v))\cap\bar{N}(u) \subseteq N(N(v))\cap N(u)=\emptyset$. Therefore,
 {\bf{N1}} holds in $\bar{\Gamma}$, as well.

To prove {\bf{N2}} it is enough to show that $\bar{\Gamma}$ is bi-transitive. Take $u_1,u_2,v_1,v_2\in \bar{V}$ with $u_1v_1,v_1u_2,u_2v_2 \in \bar{E}$. Then $u_1,u_2,v_1,v_2\in V$ and $u_1v_1,v_1u_2,u_2v_2 \in E$. As $\Gamma$ is bi-transitive, we have  $u_1v_2 \in E$. Since the edge $u_1v_2$ was not removed from $\Gamma$ to obtain $\bar{\Gamma}$, we also have $u_1v_2 \in \bar{E}$ and hence $\bar{\Gamma}$ is bi-transitive.

To prove {\bf{N3}},  take $u,v\in \bar{\Gamma}$ such that $u\cap \bar{N}(\bar{N}(v))=v\cap \bar{N}(\bar{N}(u))=\emptyset$ and $\bar{N}(u)\cap \bar{N}(v)\neq \emptyset$.
If $u\cap N(N(v))$ is non empty, then there exists $w\in V$ such that $vw,wu\in E$. Here $w\not\in\bar{\Gamma}$. Up to a rearrangement, see (iv) Proposition \ref{propA7ag2020}, we have either $w=m$, or $w=m-1$, or $m-1-r \le w \le m-2$  with $N(m-1-i)=\{m\}$ for $1\le i \le r$.
If $w=m$, then $mu\in E$ and hence $u=m-1$. If $w=m-1$, then $(m-1) u \in E$, and (iii) of Proposition \ref{propA7ag2020} yields $u=m$. If $w=m-1-i$ for some $1 \le i \le r$, then $wu \in E$. Since $N(m-i-1)=\{m\}$ this yields that $u=m$. This shows that in each of the three cases, either $u=m$, or $u=m-1$ hold, a contradiction as $m,m-1\not\in \bar{V}$. Similarly,
$v\cap N(N(u))$ cannot occur. Also, $\bar{N}(u)\cap \bar{N}(v)\neq \emptyset$ implies $N(u)\cap N(v)\neq \emptyset$. Therefore the hypotheses in {\bf{N3}} in $\Gamma$ are
satisfied, and {\bf{N3}} applies to $u,v$. Thus, $N^-(u)=N^-(v)$ whence $\bar{N}^-(u)=\bar{N}^-(v)$. Also, if $N(u)\subseteq N(v)$, then $\bar{N}(u)\subseteq \bar{N}(v)$.
Therefore, {\bf{N3}} holds true in $\bar{\Gamma}$.

Finally, assume that $\bar{\Gamma}$ contains a vertex $u$ with $\bar{N}(u)\neq \emptyset$. Then $N(u)=\{m-1-r,\ldots,m-2,m-1\}$. Therefore, if $v$ is another vertex of $\bar{\Gamma}$ without out-neighbor in $\bar{\Gamma}$, then $N(u)=N(v)$.  Since $d_iu$ and $d_iv$ are not edges of $\Gamma$, the hypotheses in {\bf{N3}} hold for $u,v$ in $\Gamma$. Therefore, $N^-(u)=N^-(v)$ whence the claim follows.
\end{proof}
We state a corollary of Proposition \ref{propB7ag2020}.
\begin{corollary}
\label{rem19ag} Assume that $\Gamma$ is a 2-cBMG which contains no two equivalent vertices. Then one of the following cases occurs.
\begin{itemize}
\item[\rm(I)] $N(u)=\{m\}$ only holds for $u=m-1$, and $\bar{\Gamma}$ is a 2-cBMG.
\item[\rm(II)] $N(u)=\{m\}$ only holds for $u=m-1,m-2$, and $\bar{\Gamma}$ is an almost 2-cBMG.
\end{itemize}
\end{corollary}

\begin{remark}
\label{31ag} {\em {The following 2-cBMG shows that its truncated graph may have two equivalent vertices

$$\Gamma:=<7\mid [1, 3], [1, 7], [2, 3], [3, 4], [4, 3], [5, 7], [6, 7], [7, 6]>.$$

In fact, Case (II) occurs with  $\bar{\Gamma}=<4\mid [1,3],[2,3],[3,4],[4,3]>$ where $\bar{\Gamma}$ is a 2-cBMG which has two equivalent vertices, namely $1$ and $2$.}}
\end{remark}

\begin{remark}
\label{rem28ag} {\em{In case (II) $m-2\not\in N(m)$ by Corollary \ref{cor1}.}}
\end{remark}
\begin{remark}\label{august24_remark2}
\em{This shows that if all 2-cBMGs on $\{1,\ldots m-2\}$ together with all almost 2-cMBGs on  $\{1,\ldots, m-3\}$ are available, then all 2-cBMGs on $\{1,\ldots m\}$ can be obtained by adding edges with head in either $\{m-1,m\}$ or $\{m-2,m-1,m\}$.}
\end{remark}

Corollary \ref{rem19ag} shows that if (I) occurs, then removing the pair $\{m-1,m\}$ of vertices from $\Gamma$ we obtain a 2-cBMG on $\{1,\ldots,m-2\}$. If (II) occurs and the hypothesis $N(m-3)\supsetneq \{m-2,m\}$ is also satisfied, then removing the triple $\{m-2,m-1,m\}$ of vertices from $\Gamma$ we obtain a 2-cBMG on $\{1,\ldots,m-3\}$. If the arising 2-cBMG contains no two equivalent vertices, then we can go on by applying Corollary \ref{rem19ag} on it. Repeating this process we end up with a partition of the vertex set of $\Gamma$ in pairs and triples, unless the hypothesis fails at some step. From (II) and Remark \ref{rem28ag}, each  triple consists of three consecutive labels, say $i+2,i+1,i$ and the subgraph of $\Gamma$ with vertex set $i+2,i+1,i$ has three edges, namely $[i+1,i+2],[i+2,i+1]$ and $[i,i+2]$. Therefore, Corollary \ref{rem19ag} suggests the following construction. Start with a partition of the vertex set of the complete bipartite digraph whose members are pairs and triples of consecutive labels, delete all edges whose endpoints belong to different members and, in each triple $i,i+1,i+2$, also one of the edges $[i+2,i]$. The arising bipartite graph is a (non-connected) 2-cBMG and will be called an \emph{elementary} 2-cBMG. From each elementary 2-cBMG we can obtain more 2-cBMGs by adding edges.

\begin{remark}\label{rem31agA}
\em{In this example we show that the procedure described may fail. Let $$\Gamma:=<7\mid  [1, 2], [1, 4], [1, 7], [2, 5], [2, 6], [3, 4], [4, 3], [5, 7], [6, 7], [7, 6]>.$$
Then Case (II) occurs with $\bar{\Gamma}=<4\mid [1,2],[1,4],[3,4],[4,3]>$ where $\bar{\Gamma}$ is a 2-cBMG without two equivalent vertices. Applying Case (I) of Corollary \ref{rem19ag} on $\bar{\Gamma}$ gives $\bar{\bar{\Gamma}}:=<2\mid [1,2]>$ that is not a 2-cBMG, only an almost 2-cBMG.}
\end{remark}

\subsection{Case n=7}
A complete bipartite graph on $\{1,\ldots, 7\}$ has exactly three partitions, namely $\Pi_1=\{7,6,5\}\cup \{4,3\}\cup \{2,1\}$, $\Pi_2=\{7,6\}\cup \{5,4,3\}\cup \{2,1\}$, and $\Pi_3=\{7,6\}\cup \{5,4\}\cup \{3,2,1\}$, each of them gives rise to two elementary 2-cBMGs with at least three vertices in each color class.
$$
\begin{array}{lll}
{\mbox{$\Pi_{11}:=<7|[7,6],[6,7],[5,7],[4,3],[3,4],[2,1],[1,2]>$ with color classes $\{6,5,3,1\}$, $\{7,4,2\}$}}\\
{\mbox{$\Pi_{12}:=<7|[7,6],[6,7],[5,7],[4,3],[3,4],[2,1],[1,2]>$ with color classes $\{6,5,4,2\}$, $\{7,3,1\}$}}\\
{\mbox{$\Pi_{21}:=<7|[7,6],[6,7],[5,4],[4,5],[3,5],[2,1],[1,2]>$ with color classes $\{7,4,3,1\}$, $\{6,5,2\}$}}\\
{\mbox{$\Pi_{22}:=<7|[7,6],[6,7],[5,4],[4,5],[3,5],[2,1],[1,2]>$ with color classes $\{7,4,3,2\}$, $\{6,5,1\}$}}\\
{\mbox{$\Pi_{31}:=<7|[7,6],[6,7],[5,4],[4,5],[3,2],[2,3],[1,3]>$ with color classes $\{6,4,2,1\}$, $\{7,5,3\}$}}\\
{\mbox{$\Pi_{32}:=<7|[7,6],[6,7],[5,4],[4,5],[3,2],[2,3],[1,3]>$ with color classes $\{7,4,2,1\}$, $\{6,5,3\}$.}}
\end{array}
$$
There exists an isomorphism between any two of the above elementary 2-cBMGs which are also color-invariant, that is, equi-colored vertices are mapped to equi-colored vertices.
By adding edges to $\Pi_{11}$ we obtain exactly seven pairwise non-isomorphic connected 2-cBMGs with color classes $\{6,5,3,1\}$, $\{7,4,2\}$ and without two equivalent vertices:
$$
\begin{array}{llll}
\Gamma_1(7):=<7\mid [1, 2], [1, 4], [2, 1], [2, 3], [3, 4], [4, 3], [5, 2], [5, 4], [5, 7], [6, 2], [6, 4], [6, 7], [7, 1], [7, 3], [7,6]>\\
\Gamma_2(7):=<7\mid [1, 2], [2, 1], [3, 4], [4, 3], [5, 2], [5, 4], [5, 7], [6, 2], [6, 7], [7, 1], [7, 6]> \\
\Gamma_3(7):=<7\mid [1, 2], [1, 7], [2, 1], [2, 6], [3, 4], [4, 3], [5, 2], [5, 7], [6, 7], [7, 6]> \\
\Gamma_4(7):=<7\mid [1, 2], [1, 4], [1, 7], [2, 1], [2, 3], [2, 5], [2, 6], [3, 4], [4, 3], [5, 7], [6, 7], [7, 6]> \\
\Gamma_5(7):=<7\mid [1, 2], [2, 1], [3, 4], [4, 3], [5, 2], [5, 4], [5, 7], [6, 7], [7, 6]> \\
\Gamma_6(7):=<7\mid [1, 2], [2, 1], [3, 4], [4, 3], [5, 2], [5, 4], [5, 7], [6, 2], [6, 4], [6, 7], [7, 1], [7, 3], [7, 6]>\\
\Gamma_7(7):=<7\mid [1, 2], [2, 1], [3, 2], [3, 4], [3, 7], [4, 1], [4, 3], [4, 5], [4, 6], [5, 2], [5, 7], [6, 7], [7, 6]>.
\end{array}
$$
\subsection{Case n=8} In this case we have three pairwise non color-invariant isomorphic elementary 2-cBMGs, $\Pi_1,\Pi_2,\Pi_3$ with at least three vertices in each color class. $\Pi_1$ has no triangle while both $\Pi_2$ and $\Pi_3$ have two triangles.
$$
\begin{array}{lll}
{\mbox{$\Pi_1=<8|[8,7],[7,8],[6,5],[5,6],[4,3],[3,4],[1,2],[2,1]>$ with color classes $\{8,6,4,2\},\{7,5,3,1\}$}}\\
{\mbox{$\Pi_2=<8|[8,7],[7,8],[6,8],[5,4],[4,5],[3,5],[1,2],[2,1]>$ with color classes $\{8,5,1\}, \{7,6,4,3,2\}$}}\\
{\mbox{$\Pi_3=<8|[8,7],[7,8],[6,8],[5,4],[4,5],[3,5],[1,2],[2,1]>$ with color classes $\{8,4,3,1\},\{7,6,5,2\}$.}}\\
\end{array}
$$
By adding edges to $\Pi_2$ we obtain exactly eighteen pairwise non-isomorphic connected 2-cBMGs with color classes  $\{8,5,1\}$, $\{7,6,4,3,2\}$ and without two equivalent vertices:
$$
\begin{array}{llll}
\hspace{-1.2cm}
 \Gamma_8(1):=<8\mid
 [1, 2], [1, 4], [1, 6], [1, 7], [2, 1], [2, 5], [2, 8], [3, 1], [3, 5], [3, 8], [4, 5], [5, 4], [6, 5], [6, 8], [7,
5], [7, 8], [8, 4], [8, 7] >\\
\hspace{-1.2cm}
 \Gamma_8(2):=<8\mid
 [1, 2], [2, 1], [3, 1], [3, 5], [4, 5], [5, 4], [6, 1], [6, 5], [6, 8], [7, 1], [7, 5], [7, 8], [8, 2], [8, 3], [8,
4], [8, 7] >\\
\hspace{-1.2cm}
 \Gamma_8(3):=<8\mid
 [1, 2], [2, 1], [3, 1], [3, 5], [4, 1], [4, 5], [5, 2], [5, 4], [6, 8], [7, 8], [8, 7] >\\
 \hspace{-1.2cm}
 \Gamma_8(4):=<8\mid
 [1, 2], [2, 1], [3, 5], [4, 5], [5, 4], [6, 8], [7, 8], [8, 7] >\\
 \hspace{-1.2cm}
 \Gamma_8(5):=<8\mid
 [1, 2], [1, 4], [1, 6], [1, 7], [2, 1], [2, 5], [2, 8], [3, 1], [3, 5], [3, 8], [4, 5], [4, 8], [5, 4], [5, 6], [5,
7], [6, 8], [7, 8], [8, 7] >\\
\hspace{-1.2cm}
 \Gamma_8(6):=<8\mid
 [1, 2], [2, 1], [3, 5], [4, 5], [5, 4], [6, 1], [6, 5], [6, 8], [7, 8], [8, 7] >\\
 \hspace{-1.2cm}
 \Gamma_8(7):=<8\mid
 [1, 2], [1, 6], [1, 7], [2, 1], [2, 8], [3, 5], [4, 5], [5, 4], [6, 8], [7, 8], [8, 7] >\\
 \hspace{-1.2cm}
 \Gamma_8(8):=<8\mid
 [1, 2], [2, 1], [3, 1], [3, 5], [3, 8], [4, 5], [5, 4], [6, 5], [6, 8], [7, 8], [8, 7] >\\
\hspace{-1.2cm}
 \Gamma_8(9):=<8\mid
 [1, 2], [2, 1], [3, 1], [3, 5], [3, 8], [4, 5], [4, 8], [5, 4], [5, 6], [5, 7], [6, 8], [7, 8], [8, 7] >\\
 \hspace{-1.2cm}
 \Gamma_8(10):=<8\mid
 [1, 2], [2, 1], [3, 1], [3, 5], [3, 8], [4, 1], [4, 5], [5, 2], [5, 4], [6, 8], [7, 8], [8, 7] >\\
\hspace{-1.2cm}
 \Gamma_8(11):=<8\mid
 [1, 2], [1, 3], [1, 4], [1, 6], [1, 7], [2, 1], [2, 5], [2, 8], [3, 5], [4, 5], [5, 4], [6, 5], [6, 8], [7, 8], [8,
7] >\\
\hspace{-1.2cm}
 \Gamma_8(12):=<8\mid
 [1, 2], [1, 4], [2, 1], [2, 5], [3, 1], [3, 5], [4, 5], [5, 4], [6, 1], [6, 5], [6, 8], [7, 8], [8, 7] >\\
 \hspace{-1.2cm}
 \Gamma_8(13):=<8\mid
 [1, 2], [1, 3], [1, 4], [1, 6], [1, 7], [2, 1], [2, 5], [2, 8], [3, 5], [4, 5], [5, 4], [6, 8], [7, 8], [8, 7] >\\
 \hspace{-1.2cm}
  \Gamma_8(14):=<8\mid
 [1, 2], [2, 1], [3, 5], [3, 8], [4, 5], [5, 4], [6, 8], [7, 8], [8, 7] >\\
 \hspace{-1.2cm}
 \Gamma_8(15):=<8\mid
 [1, 2], [2, 1], [3, 1], [3, 5], [3, 8], [4, 1], [4, 5], [4, 8], [5, 2], [5, 4], [5, 6], [5, 7], [6, 8], [7, 8], [8,
7] >\\
\hspace{-1.2cm}
 \Gamma_8(16):=<8\mid
 [1, 2], [2, 1], [3, 5], [4, 5], [5, 4], [6, 1], [6, 8], [7, 8], [8, 7] >\\
 \hspace{-1.2cm}
 \Gamma_8(17):=<8\mid
 [1, 2], [2, 1], [3, 5], [3, 8], [4, 5], [4, 8], [5, 4], [5, 6], [5, 7], [6, 8], [7, 8], [8, 7] >\\
 \hspace{-1.2cm}
 \Gamma_8(18):=<8\mid
 [1, 2], [1, 3], [1, 4], [1, 6], [1, 7], [2, 1], [2, 5], [2, 8], [3, 5], [3, 8], [4, 5], [4, 8], [5, 4], [5, 6], [5,
7], [6, 8], [7, 8], [8, 7]>.
\end{array}
$$
\subsection{2-cBMGs containing two symmetric edges with a common vertex}
Let $\bar{V}$ be the set of all vertices of $\Gamma=\Gamma(V,E)$ which are shared by at least two symmetric edges of $\Gamma$. Consider the undirected graph $\Sigma=\Sigma(\bar{V},\bar{E})$ where $\bar{u}\bar{v}\in \bar{E}$ if and only if $\bar{u}\bar{v}$ is a symmetric edge of $\Gamma$.
\begin{proposition}
\label{pro28oct} Every connected component of $\Sigma$ is a complete bipartite graph.
\end{proposition}
 \begin{proof} Let $\Omega$ be a connected component of $\Sigma$. Obviously, $\Omega$ is a bipartite graph. Let $\bar{U}$ and $\bar{W}$ denote its color classes. Then $\Omega=\Omega(\bar{U}\cup \bar{W},\bar{F})$.
 For any two vertices $\bar{u}\in \bar{U}$ and $\bar{w}\in \bar{W}$ we have to show that $\bar{u}\bar{w}\in \bar{F}$. Since $\Omega$ connected, there is a path $\bar{u}=\bar{u}_0\bar{u}_1\cdots\bar{u}_n=\bar{w}$.
 Since $\bar{u}_i\bar{u}_{i+1}$, $i=0,\ldots,n-1$, is a symmetric edge of $\Gamma$, $\bar{u}=\bar{u}_0\rightarrow \bar{u}_1\cdots\rightarrow \cdots\bar{u}_n=\bar{w}$ is a directed path in $\Gamma$. By $\bf{N}2$,
 this yields $\bar{u}\bar{w}\in E$. Also, $\bar{w}=\bar{u}_n\rightarrow \bar{u}_{n-1}\cdots\rightarrow \cdots\bar{u}_0=\bar{u}$ is a directed path in $E$, and hence $\bar{w}\bar{u}\in E$. Therefore, $\bar{u}\bar{w}$ is a symmetric edge of $\Gamma$. Thus $\bar{u}\bar{w}\in \bar{F}$.
 \end{proof}
 Propositions \ref{pro28oct} and \ref{prop23ag} suggest how to construct a large family of 2-cBMGs containing many symmetric edges with common vertices. Take pairwise disjoint complete bipartite graphs $\Lambda_i=(U_i\cup W_i,E_i)$ with $i=1,\ldots,m$. Let $U=\cup_{i=1}^m U_i,W=\cup_{i=1}^m W_i$. Then $\Gamma=\Gamma(U\cup W,E)$ is defined to be the digraph whose edge set $E=F\cup G$ where $F$ consists of all (symmetric) edges of $\Lambda_i$ with $1\le i \le m$ while $G$ consists of all (non-symmetric) edges $uw$ where either $u\in U_1$, $w\in W_i$, or $w\in W_1$, $u\in U_i$, and $2\le i \le m$. Clearly, $\Gamma$ is a bipartite graph with color classes $U$ and $W$. 
 \begin{proposition}
 $\Gamma$ is a 2-cBMG.
 \end{proposition}
 \begin{proof}
 From the construction of $\Gamma$  it is clear that $\bf{N}4$ holds. 
 
 To prove  $\bf{N}2$, take four vertices $u_1,u_2\in U$ and $w_1,w_2\in W$ such that $u_1w_1,w_1u_2,u_2w_2\in E$. We have to show $u_1w_2\in E$. From the construction, this holds for $u_1\in U_1$. Therefore, $u_1\in U_i$ with $i\ge 2$. Then $w_1\in W_i$. This yields $u_2\in U_i$ whence $w_2\in W_i$ follows. Thus $u_1,u_2,w_1,w_2$ are vertices of the complete bipartite graph $\Lambda_i$. Hence $u_1w_2\in E_i$, and by the construction, $u_1w_2\in E$.

 To prove $\bf{N}1$, take two vertices $u,v$ of $\Gamma$ that satisfy the hypothesis of $\bf{N}1$. Then neither $u$ nor $v$ is a vertex of $\Lambda_1$. Furthermore, the claim of $\bf{N}1$ clearly holds if $u$ and $v$ have the same color. Therefore, we may assume that $u\in U_i$ and $v\in W_j$ for some $2\le i,j \le m$.  Since $\Lambda_i$ is a complete bipartite graph, $i\neq j$ holds otherwise $u$ and $v$ would not be independent. In particular, $j\geq 2$. Hence $N(v)=U_j$ from which $N(N(v))=W_j$ follows. For $i\ge 2$, this argument shows that $N(u)=W_i$ and $N(N(u))=U_i$. Thus,  $N(v)\cap N(N(u))=U_j\cap U_i=\emptyset$ and $N(u)\cap N(N(v))=W_i\cap W_j=\emptyset$, whence the claim of $\bf{N}1$ follows. 
 
 Finally, we show that no two vertices $u,v$ of $\Gamma$ satisfy the hypotheses of $\bf{N}3$. By way of a contradiction, let $u,v$ be two vertices of $\Gamma$ that satisfy the hypotheses of $\bf{N}3$. Clearly, $u$ and $v$ have the same color. Without loss of generality, $u,v\in U$. Let $u\in U_i$, $v\in U_j$. From the hypotheses of $\bf{N}3$, $i\neq j$, and  we may assume that $1\le i< j \le m$. For $i\geq 2$, no vertex $w\in W$ exists for which either $uw,wv\in E$ or $wu,vw\in E$ and hence $u,v$ do not satisfy the hypotheses of $\bf{N}3$. For $i=1$, the only vertices $w\in W$ with $uw,wv\in E$ or $wu,vw\in E$ are those in $W_j$. But then $wv\in W_j$, and again the hypotheses of $\bf{N}3$ are not satisfied by $u,v$. 
 \end{proof}

\section{Computational classification of small 2-cBMGs}\label{ccs}
A straightforward exhaustive search of 2-cBMGs with the use of the MAGMA Algebra package \cite{magma} is possible for $2\le n \le 7$ since all subgraphs of the complete digraphs on $n$ vertices can be stoblue for $n\le 7$. By Remarks \ref{august24_remark1} and \ref{august24_remark2}, the resulting data-base can be used for the effective construction of larger 2-cBMGs.

Let $n\ge 3$ be an integer. For $2\le i \le n-1$, let $K_{i,n-i}$ be the complete bipartite digraph with color classes $B_1=\{1,\ldots i\}$ and $B_2=\{i+1,\ldots n\}$. Up to relabeling and interchanging $B_1$ with $B_2$, we may assume $i\le n/2$.
\begin{itemize}
\item $A(n,i)$:=set all pairwise non-isomorphic subgraphs $K_{i,n-i}$ satisfying {\bf{N1}}, {\bf{N2}}, {\bf{N3}}.
\item $B(n,i)$:=set of all  connected digraphs in $A(n,i)$.
\item $C(n,i)$:=set of all  digraphs in $A(n,i)$ which contain no two equivalent vertices.
\item $D(n,i)$:=set of all digraphs in $A(n,i)$ containing no vertex without out-neighbor.
\item $E(n,i)=B(n,i)\cap C(n,i)\cap D(n,i)$. 
\end{itemize}
\subsection{Case n=3} The unique value for $i$ is $i=1$, that is, $B_1=\{1\}$ and $B_2=\{2,3\}$. Then
\begin{itemize}
\item $\Gamma_1(3):=<3\mid[1,2],[1,3],[3,1]>$.
\end{itemize}

\subsection{Case n=4} The unique value for $i$ is $i=2$, that is, $B_1=\{1,2\}$ and $B_2=\{3,4\}$. Then
$$|A(4,2)|=26, |B(4,2)|=14, |C(4,2)|=5, |D(4,2)|=11, |E(4,2)|=2.$$
Moreover, $E(4,2)$ consists of the following two 2-cBMGs:
\begin{itemize}
\item $\Gamma_1(4):=<4\mid[1,4],[2,3],[2,4],[3,1],[3,2],[4,1]>$
\item $\Gamma_2(4):=<4\mid [1,4],[2,4],[3,1],[3,2],[4,1]>$.
\end{itemize}
\subsection{Case n=5} The unique value for $i$ is $i=2$, that is, $B_1=\{1,2\},B_2=\{3,4,5\}$. Then $$|A(5,2)|=122, |B(5,2)|=74, |C(4,2)|=51, |D(5,2)|=16, |E(5,2)|=4.$$
Moreover, $E(5,2)$ consists of the following four 2-cBMGs:
\begin{itemize}
\item $\Gamma_1(5):=<5\mid [1, 4], [1, 5], [2, 4], [3, 1], [3, 2], [4, 2], [5, 2] >$
\item $\Gamma_2(5):=<5\mid [1, 3], [2, 5], [3, 1], [4, 1], [4, 2], [5, 2]>$
\item $\Gamma_3(5):=<5\mid [1, 4], [2, 3], [2, 4], [2, 5], [3, 1], [4, 1], [5, 1], [5, 2]>$
\item $\Gamma_4(5):=<5\mid [1, 4], [2, 3], [2, 4], [3, 1], [3, 2], [4, 1], [5, 1], [5, 2]>$
\end{itemize}
\subsection{Case n=6} There are two values for $i$ namely $i=2,3$.
\subsubsection{$B_1=\{1,2\},B_2=\{3,4,5,6\}$} In this case,  $$|A(6,2)|=353, |B(6,2)|=175, |C(6,2)|=69, |D(6,2)|=33, |E(6,2)|=2.$$
Moreover, $E(4,2)$ consists of the following two 2-cBMGs:
\begin{itemize}
\item $\Gamma_1(6):=<6\mid [1, 4],[2, 4], [2, 5], [2, 6], [3, 1], [3, 2], [4, 1], [5,1], [5, 2],[6,1]>$
\item $\Gamma_2(6):=<6\mid [1, 3], [2, 5], [3, 1], [4, 1], [4, 2] [5, 2],[6,1] >$.
\end{itemize}
\subsubsection{$B_1=\{1,2,3\},B_2=\{4,5,6\}$} In this case, $$|A(6,3)|=347, |B(6,3)|=172, |C(6,3)|=149, |D(6,3)|=33, |E(6,3)|=8.$$
Moreover, $E(6,3)$ consists of the following eight 2-cBMGs:
\begin{itemize}
\item $\Gamma_3(6):=<6\mid [1, 6], [2, 4], [2, 6], [3, 6], [4, 1], [4, 2], [4, 3], [5, 1], [5, 2], [5,3], [6, 3]>$
\item $\Gamma_4(6):=<6\mid  [1, 5], [2, 4], [2, 5], [2, 6], [3, 6], [4, 1], [5, 1], [6, 3]>$
\item $\Gamma_5(6):=<6\mid  [1, 4], [1, 5], [1, 6], [2, 5], [3, 5], [3, 6], [4, 1], [4, 2], [4, 3], [5,
2], [6, 2], [6, 3]>$
\item $\Gamma_6(6):=<6\mid  [1, 5], [1, 6], [2, 5], [3, 5], [3, 6], [4, 1], [4, 2], [4, 3], [5, 2], [6,
2], [6, 3] >$
\item $\Gamma_7(6):=<6\mid 1, 4], [1, 5], [1, 6], [2, 4], [2, 6], [3, 6], [4, 3], [5, 1], [5, 2], [5,
3], [6, 3]>$
\item $\Gamma_8(6):=<6\mid [1, 4], [1, 6], [2, 6], [3, 6], [4, 2], [4, 3], [5, 1], [5, 2], [5, 3], [6,
3]>$
\item $\Gamma_9(6):=<6\mid [1, 4], [1, 6], [2, 4], [3, 6], [4, 2], [5, 1], [5, 2], [5, 3], [6, 3]>$
\item $\Gamma_{10}(6):=<6\mid [1, 5], [2, 4], [2, 5], [2, 6], [3, 6], [4, 1], [4, 2], [4, 3], [5, 1], [6,
3]>$.
\end{itemize}
\subsection{Case n=7} There are two values for $i$ namely $i=2,3$.
\subsubsection{$B_1=\{1,2\},B_2=\{3,4,5,6,7\}$} In this case, $$|A(7,3)|=647, |B(7,3)|=283, |C(7,3)|=571, |D(7,3)|=59, |E(7,3)|=1.$$
Moreover, $E(7,3)$ consists of just one 2-cBMG, namely $\Gamma_1(7)$; see Remark \ref{31ag}.

\subsubsection{$B_1=\{1,2,3\},B_2=\{4,5,6,7\}$} In this case, $$|A(7,3)|=555, |B(7,3)|=324, |C(7,3)|=352, |D(7,3)|=126, |E(7,3)|=21.$$
Moreover, $E(7,3)$ consists of the seven 2-cBMGs given in Section \ref{2cbmg} together with the following fourteen 2-cBMGs:
\begin{itemize}
\item $\Gamma_8(7):=<7\mid  [1, 4], [2, 4], [2, 5], [3, 7], [4, 1], [5, 1], [6, 1], [6, 2], [6, 3], [7, 3]>$
\item $\Gamma_9(7):=<7\mid  [1, 4], [2, 7], [3, 7], [4, 1], [5, 1], [6, 1], [6, 2], [6, 3], [7, 3]>$
\item $\Gamma_{10}(7):=<7\mid [1, 4], [2, 4], [3, 4], [3, 5], [3, 7], [4, 1], [5, 1], [5, 2], [6, 1], [6, 2], [6, 3], [7, 1], [7, 2], [7, 3]>$
\item $\Gamma_{11}(7):=<7\mid [1, 4], [2, 4], [2, 5], [3, 4], [3, 5], [3, 7], [4, 1], [5, 1], [6, 1], [6, 2], [6, 3], [7, 1], [7, 2], [7, 3]>$
\item $\Gamma_{12}(7):=<7\mid  [1, 4], [2, 4], [2, 5], [2, 6], [2, 7], [3, 7], [4, 1], [5, 3], [6, 1], [7, 3]>$
\item $\Gamma_{13}(7):=<7\mid [1, 4], [2, 4], [2, 6], [2, 7], [3, 7], [4, 1], [5, 1], [5, 2], [5, 3], [6, 3], [7, 3]>$
\item $\Gamma_{14}(7):=<7\mid [1, 4], [2, 4], [2, 5], [2, 7], [3, 4], [3, 7], [4, 1], [5, 1], [5, 3], [6, 1], [6, 2], [6, 3], [7, 1], [7, 3]>$
\item $\Gamma_{15}(7):=<7\mid [1, 5], [2, 7], [3, 7], [4, 2], [4, 3], [5, 1], [6, 1], [6, 2], [6, 3], [7,3]>$
\item $\Gamma_{16}(7):=<7\mid [1, 4], [1, 5], [1, 6], [2, 6], [3, 4], [3, 5], [3, 6], [3, 7], [4, 1], [4, 2], [5, 2], [6, 2], [7, 1], [7, 2], [7,3]>$
\item $\Gamma_{17}(7):=<7\mid  [1, 4], [2, 4], [2, 6], [2, 7], [3, 7], [4, 1], [5, 1], [5, 2], [5, 3], [6, 1], [6, 3], [7, 3]>$
\item $\Gamma_{18}(7):=<7\mid [1, 4], [2, 4], [2, 5], [2, 6], [2, 7], [3, 7], [4, 1], [5, 3], [6, 1], [6, 3], [7, 3]>$
\item $\Gamma_{19}(7):=<7\mid  [1, 4], [2, 4], [3, 7], [4, 1], [5, 1], [5, 2], [5, 3], [6, 1], [6, 2], [7, 3]>$
\item $\Gamma_{20}(7):=<7\mid[1, 4], [1, 5], [1, 6], [1, 7], [2, 6], [2, 7], [3, 7], [4, 1], [4, 2], [4, 3], [5, 2], [5, 3], [6, 3], [7, 3] >$
\item $\Gamma_{21}(7):=<7\mid[1, 4], [2, 4], [2, 5], [2, 7], [3, 4], [3, 5], [3, 7], [4, 1], [5, 1], [6, 1], [6, 2], [6, 3], [7, 1], [7, 3] >$.
\end{itemize}

\end{document}